\documentclass[12pt,a4paper,english,american]{article}
\usepackage{helvet}
\usepackage[top=2.5cm, bottom=2.5cm, left=3.0cm, right=2.5cm] {geometry}
\usepackage{babel}
\usepackage{verbatim}
\usepackage{amsthm}
\usepackage{amsmath}
\usepackage{amssymb}
\usepackage[unicode=true,pdfusetitle,
 bookmarks=true,bookmarksnumbered=false,bookmarksopen=false,
 breaklinks=false,pdfborder={0 0 1},backref=false,colorlinks=false]
 {hyperref}

\makeatletter

\numberwithin{equation}{section}
\numberwithin{figure}{section}
 \theoremstyle{definition}
 \newtheorem*{defn*}{\protect\definitionname}
\theoremstyle{plain}

  \theoremstyle{remark}
  \theoremstyle{plain}
  
  \theoremstyle{plain}
  
  \theoremstyle{plain}
  
  \theoremstyle{definition}

\usepackage[T1]{fontenc}    % Silbentrennung auch nach Umlauten m�glich
\usepackage{amsmath}
\usepackage{enumerate}

\usepackage{euscript}
\usepackage{amsfonts}

\theoremstyle{definition}

\usepackage{color}

\theoremstyle{definition}
\newtheorem{definition}{Definition}%[section]
\newtheorem{theorem}[definition]{Theorem}
\newtheorem{remark}[definition]{Remark}
\newtheorem{proposition}[definition]{Proposition}

\newcounter{enumctr}

\DeclareMathAccent{\Circ}{\mathalpha}{operators}{"17}

\newcommand{\cC}{\mathcal{C}}

\newcommand{\R}{\mathbb{R}}

\makeatother

  \addto\captionsamerican{\renewcommand{\corollaryname}{Corollary}}
  \addto\captionsamerican{\renewcommand{\definitionname}{Definition}}
  \addto\captionsamerican{\renewcommand{\lemmaname}{Lemma}}
  \addto\captionsamerican{\renewcommand{\propositionname}{Proposition}}
  \addto\captionsamerican{}
  \addto\captionsamerican{\renewcommand{\theoremname}{Theorem}}
  \addto\captionsenglish{\renewcommand{\corollaryname}{Corollary}}
  \addto\captionsenglish{\renewcommand{\definitionname}{Definition}}
  \addto\captionsenglish{\renewcommand{\lemmaname}{Lemma}}
  \addto\captionsenglish{\renewcommand{\propositionname}{Proposition}}
  \addto\captionsenglish{}
  \addto\captionsenglish{\renewcommand{\theoremname}{Theorem}}
  \providecommand{\corollaryname}{Corollary}
  \providecommand{\definitionname}{Definition}
  \providecommand{\lemmaname}{Lemma}
  \providecommand{\propositionname}{Proposition}
  
\providecommand{\theoremname}{Theorem}

\begin{document}
\selectlanguage{english}%

%%%%%%%%%%%%%%%%%%%%%%%%%%%%%%%%%%%

\title{Generation of random dynamical systems from fractional stochastic delay differential equations}

\author{Luu Hoang Duc, Bj{\"o}rn Schmalfuss (FSU-Jena), Stefan Siegmund}

\titlepage

\selectlanguage{american}%
\setcounter{section}{0}
\date{Submitted January 24, 2014}
%\date{\today}

\title{Generation of random dynamical systems from fractional stochastic delay differential equations}

\author{Luu Hoang Duc\\
Institute of Mathematics, VAST, Vietnam
\& \\
Institute of Analysis, 
Technische Universit\"at Dresden\\[3ex]
Stefan Siegmund \\
Institute of Analysis\\
Technische Universit\"at Dresden\\[3ex]
        Bj{\"o}rn Schmalfuss\\
        Institute of Stochastics\\
        Friedrich-Schiller-Universit{\"a}t Jena }
\maketitle

%%%%%%%%%%%%%%%%%%%%%%%%%%%%%%%%%%%

\begin{abstract}
In this note we prove that a fractional stochastic delay differential equation which satisfies natural regularity conditions generates a continuous random dynamical system on a subspace of a H{\"o}lder space which is separable.
\end{abstract}
\vspace{1ex} \noindent {\bf Key words and phrases:} fractional Brownian motion, stochastic differential equations, stochastic delay differential equations, stochastic functional differential equations.
\\[1ex]

%
%%%%%%%%%%%%%%%%%%%%%%%%%%%%%%%%%%%%%%%%%%%%%%%%%%%%%%%%%%%%

\section{Introduction}

Fractional Brownian motion (fBm) is a family of centered Gaussian processes $B^H = \{B^H(t)\}$, $t\in \R$ or $\R_+$, indexed by the Hurst parameter $H \in (0,1)$ with continuous sample paths and the covariance function
\[
R_H(s,t) = \tfrac{1}{2}(t^{2H} + s^{2H} - |t-s|^{2H}).
\]
It is a self-similar process with stationary increments and has a long memory when $H > \frac{1}{2}$ (see Mandelbrot and van Ness \cite{mandelbrot}, or Beran \cite{beran}).

In the last decade, stochastic differential equations driven by fractional Brownian motions (in short fSDE) have attracted a lot of research interest (see \cite{biagini, nualart1, nualart2, nualart3, bou1, bou2, bou3, atienza1, atienza2, chen} and the references therein). Since $B^H$ is not a semimartingale if $H \ne \frac{1}{2}$, we cannot apply the classical Ito theory to construct a stochastic integral w.r.t.\ the fBm by taking the limit in the sense of probability convergence of a sequence of Darboux sums. In contrast, the stochastic integral w.r.t.\ the fBm can be defined as a point-wise limit using the so called {\it rough path theory} as seen in Friz and Victoir \cite{friz}, Coutin and Qian \cite{coutin}, or {\it fractional calculus theory}, as seen in Samko et al. \cite{samko}, Z{\"a}hle \cite{zahle}. For this approach, the theory of stochastic differential equations driven by the fBm has been developed intensively by Nualart and R{\u{a}}{\c{s}}canu \cite{nualart3} for finite dimensional spaces, and Maslowski and Nualart \cite{nualart2}, Hu and Nualart \cite{nualart1} for infinite dimensional spaces.

Recently, stochastic functional differential equations driven by fractional Brownian motions (or in short fSFDE) have been studied by several authors. Boufoussi and Hajji \cite{bou1} proved the existence and uniqueness theorem for stochastic delay differential equations driven by fBm (fSDDE) in a finite dimensional space, and then extended the results for systems in a separable Hilbert space in Boufoussi et al.\ \cite{bou2,bou3}. They also proved that the solution of an fSDDE is continuous w.r.t.\ the initial values in the phase space.

One important issue in studying dynamics of stochastic differential equations (SDE) is to check whether or not it generates a random dynamical system (RDS). This issue is presented  for SDE in Arnold \cite{arnold}, and then proved by Mohammed and Scheutzow \cite{moham2} for a class of stochastic delay differential equations (SDDE) which satisfy some regularity conditions. It is worth mentioning that stochastic delay differential equations driven by a Wiener process in general do not generate a continuous RDS. For example, consider the following simple one-dimensional equation
\begin{eqnarray}\label{counterexample}
  dx(t)&=& x(t-1) dW(t),\\ \nonumber
  x(0) &=& v \in \R,\ x_0 = \eta\in L^2([-1,0],\mathbb{R}).
\end{eqnarray}
We know from Mohammed \cite[p.\ 144]{moham1} that this SDDE does not generate a continuous random dynamical system on $\mathcal{M} = \R \times L^2([-r,0],\R)$. Even the solution of \eqref{counterexample} does not depend continuously and linearly on the initial state $\eta \in L^2([-r,0],\R)$.

In this paper, we follow the technique developed by Boufoussi et al.\ \cite{bou1} to study a class of fSDDE in which the coefficient functions are time independent. An important remark here is that, unlike for the SDDE case in which the usual phase space $\mathcal{M} = \R \times L^2([-r,0],\R^d)$ is a separable Banach space, in the context of fSDDE, the phase space is often a H{\"o}lder space of the form $C^{1-\alpha}([-r,0], \R^d)$ and is therefore not separable. This difference makes it challenging to prove the measurability and even impossible to prove the continuity of the cocycle with respect to the time argument, thus it is very hard to apply \cite{castaing} to prove the measurability of the cocycle, see Remark \ref{rem1}. Therefore, we will have to restrict our consideration to a subspace $C^{0,1-\alpha}([-r,0],\R^d)$ of $C^{1-\alpha}([-r,0],\R^d)$ which is separable (see Friz and Victoir \cite{friz}). As we show in Theorem 1, this class of fSDDE then generates a continuous random dynamical system.

%%%%%%%%%%%%%%%%%%%%%%%%%%%%%%%%%%%%%%%%%%%%%%%%%%%%%%%%%%%%%%%%

\section{Preliminaries on random dynamical systems}

Let $(\Omega,\mathcal{F},\mathbb{P})$ be a probability space. On this probability space we consider a measurable flow $\theta$
\[
\theta:\mathbb{R}\times\Omega\to\Omega
\]
such that $\theta_t(\cdot): \Omega \to \Omega$ is $\mathbb{P}$-preserving, i.e.\ $\mathbb{P}(\theta_t^{-1}(A)) = \mathbb{P}(A)$ for every $A \in \mathcal{F}$, $t \in \R$, and $(\theta_t)_{t\in \mathbb{R}}$ satisfies the group property, i.e.\ $\theta_{t+s} = \theta_t \circ \theta_s$ for all $t,s \in \mathbb{R}$. A general model for noise is the quadruple $(\Omega,\mathcal{F},\mathbb{P},(\theta_t)_{t\in \mathbb{R}})$ which is called a {\em metric dynamical system}.

An example of a metric dynamical system is given as follows: Let $\Omega$ denote the space of all continuous functions $\omega: \mathbb{R} \to \mathbb{R}^m$ such that $\omega(0) = 0$; $\mathcal{F}$ the Borel $\sigma$-algebra of $\Omega$ generated by the compact open topology; $\mathbb{P}$ the Wiener measure on $\mathcal{F}$ generated by a fractional Brownian motion $(W_t)_{t\in \R}$ on $\R^m$. For each $t \in \R$, construct the {\em Wiener shift} $\theta_t: \Omega \to \Omega$, i.e.\ $\theta_t \omega(\cdot) = \omega(t+ \cdot) - \omega(t)$. Then $(\theta_t)_{t\in\R}$ is $\mathbb{P}$-preserving and satisfies the group property. In a similar manner, we can construct a metric dynamical system for fractional Brownian motion. In particular, we only need to replace the Wiener measure by the Gau{\ss}ian measure generated by fractional Brownian motion. Indeed, the shift operators $\theta_t$ for $t\in\R$ preserve this measure which follows by the homogeneity of the increments of the fractional Brownian motion, see Biagini et al.\ \cite[p.\ 5]{biagini}. A fractional Brownian motion with Hurst parameter $H$ has a $\beta$-H{\"o}lder continuous version for any fixed $\beta\in (0,H)$. For our purpose we need a metric dynamical system having H{\"o}lder continuous paths. Let us denote the subset of $\Omega$ consisting of paths which are $\beta$-H{\"o}lder continuous on any non-trivial compact interval of $\R$ by $\Omega_\beta$. Note that this set is invariant with respect to $(\theta_t)_{t\in\R}$. Let us consider the trace-$\sigma$-algebra
$\mathcal{F}_\beta=\mathcal{F}\cap \Omega_\beta$ and let $\mathbb{P}$ be restricted to this $\sigma$-algebra. It is not hard to see that
the restriction of $\theta$ to $\Omega_{\beta}\times \R$ is $\mathcal{F}_\beta\otimes\mathcal{B}(\R),\mathcal{F}_\beta$-measurable, see Caraballo et al. \cite{caduluschm10}. Hence we can consider a metric dynamical system with H{\"o}lder continuous path which we are going to use in the following. However, for brevity we will keep for this new metric dynamical system the old notation $(\Omega,\mathcal{F},\mathbb{P},\theta)$.

Let $X$ be a Banach space. Then we define an RDS as a  measurable mapping
\[
\varphi:\R^+\times \Omega\times X\to X
\]
satisfying the cocycle property
\begin{eqnarray*}
\varphi(t+s,\omega,x)&=&\varphi(t,\theta_s\omega,\cdot)\circ\varphi(s,\omega,x)\qquad \text{for all }
t,\,s\in\mathbb{R}_+,\,\omega\in\Omega,\,x\in X,\qquad \\
\varphi(0,\omega)&=&{\rm id}_X.
\end{eqnarray*}
An RDS $\varphi$ is called {\em continuous} if each mapping $x \mapsto \varphi(t,\omega,x)$ is continuous. If $\Omega$ consists only of one element so that $\theta_s = id$ for all $s\in \R$, then such $\omega \in \Omega$ can be neglected and $\varphi$ is indeed a semi-group and has probability one.

%%%%%%%%%%%%%%%%%%%%%%%%%%%%%%%%%%%%%%%%%%%%%%%%%%%%%%%%%%%%%%%%

\section{Stochastic integrals with respect to the fractional Brownian motion}

Before considering the main problem, it is necessary to define the stochastic integral in the sense of the generalized integration by parts formula. We should mention here the fundamental work of Young \cite{young} which allows to define a kind of Riemann-Stieltjes integral for H{\"o}lder continuous integrands and integrators (see also Z{\"a}hle \cite{zahle} for the integation by parts method). Indeed, we first introduce function spaces
\begin{eqnarray*}
C^{\nu}([-r,T],\R^d) &=& \Big\{f: [-r,T] \to \R^d, \|f\|_{\nu} := \|f\|_{\infty} + \sup \limits_{-r\leq s < t\leq T} \tfrac{\|f(t)-f(s)\|}{|t-s|^{\nu}} < \infty \Big\}
\\
W^{\alpha,1}_0 ([0,T],\R^d) &=& \Big\{f: [0,T] \to \R^d, \|f\|_{\alpha,1} := \int_0^T \Big( \tfrac{\|f(s)\|}{s^{\alpha}} + \int_0^s \tfrac{\|f(s)-f(u)\|}{|s-u|^{1+\alpha}} du \Big) ds < \infty \Big\}
\\
W^{1-\alpha,\infty}_T([0,T],\R^d) &=& \Big \{g: [0,T] \to \R^d, \|g\|_{1-\alpha,\infty,T} := \sup \limits_{0< s< t< T} \Big(\tfrac{\|g(t)-g(s)\|}{|t-s|^{1-\alpha}} + {} \\
& & {} + \int_s^t \tfrac{\|g(u)-g(s)\|}{|u-s|^{2-\alpha}} du\Big) < \infty \Big\}.
\end{eqnarray*}
For each fixed $\frac{1}{2}< \nu< H$, choose $\alpha \in (1-\nu, \frac{1}{2})$.
It can be proved (see Nualart and R{\u{a}}{\c{s}}canu \cite{nualart3} or Boufoussi and Hajji \cite{bou1}) that
\begin{equation}\label{prop2}
C^{\nu}([0,T],\R^d) \subset W^{1-\alpha,\infty}_T([0,T],\R^d) \subset C^{1-\alpha} ([0,T],\R^d)\subset W^{\alpha,1}_0([0,T],\R^d).
\end{equation}
Since $B^H$ is an fBm, each trajectory $B^H(\cdot,\omega) = \omega(\cdot)$ belongs to $C^{\nu}([0,T],\R)$. Thus for each $f \in W^{\alpha,1}_0([0,T],\R^d)$ and each trajectory $\omega \in W^{1-\alpha,\infty}_T([0,T],\R)$ we can define fractional derivatives
\begin{eqnarray*}
D^{\alpha}_{a^+} f(s) &:=& \frac{1}{\Gamma(1-\alpha)} \Big(\frac{f(s)}{(s-a)^{\alpha}} + \alpha \int_a^s \frac{f(s)-f(u)}{(s-u)^{1+\alpha}}du \Big) \\
D^{1-\alpha}_{t^-} \omega(s) &:=& \frac{(-1)^{1-\alpha}}{\Gamma(\alpha)} \Big(\frac{\omega(s)}{(t-s)^{1-\alpha}} + (1-\alpha) \int_s^t \frac{\omega(s)-\omega(u)}{(s-u)^{2-\alpha}}du \Big),
\end{eqnarray*}
$0\le a <t\le T$. Following Nualart and R{\u{a}}{\c{s}}canu \cite{nualart3}, Z{\"a}hle \cite{zahle}, we define
\[
\int_{a}^{b} fd w := (-1)^{\alpha} \int_{a}^{b} D^{\alpha}_{a^+} f(s) D^{1-\alpha}_{b^-}\omega_{b^-}(s)ds,
\]
where $\omega_{b^-}(s) = \omega(s)-\omega(b)$. Additionally, by defining
\[
\Lambda_{\alpha}(g) := \frac{1}{\Gamma(1-\alpha)} \sup \limits_{0<s<t<T} |D^{1-\alpha}_{t^-} g_{t^-}(s)|,
\]
we have the estimate
\[
\Big \|\int_0^t f d\omega \Big \| \leq \Lambda_{\alpha}(\omega) \|f\|_{\alpha,1} \leq \frac{1}{\Gamma(1-\alpha) \Gamma(\alpha)} \|\omega\|_{1-\alpha,\infty,T} \|f\|_{\alpha,1}.
\]
For more details, see Nualart and R{\u{a}}{\c{s}}canu \cite{nualart3} and Boufoussi and Hajji \cite{bou1}.

Motivated by Friz and Victoir \cite[Theorem 5.33]{friz}, we introduce the following subspace of $C^{\beta}([-r,T],\R^d)$, for $\beta\in (0,1]$
\begin{equation}\label{eq0}
C^{0,\beta}([-r,T],\R^d) := \Big\{\eta \in C^{\beta}([-r,T],\R^d): \lim \limits_{\delta \to 0} \|\eta\|_{\beta,\delta,-r,T} = 0\Big\},
\end{equation}
where
\[
 \|\eta\|_{\beta,\delta,-r,T}:= \sup \limits_{\substack{-r\leq s < t\leq T,\\ |t-s|< \delta }} \tfrac{\|\eta(t)-\eta(s)\|}{|t-s|^{\beta}}
\]

By Proposition 5.38 in Friz and Victoir \cite{friz}, $C^{0,\beta}([-r,T],\R^d)$ is a separable Banach space.

Now consider the fSDDE in the differential form
\begin{eqnarray}\label{eq01}
dX(t) &=& F(X_t)dt + G(X_t)dB^H(t), \  t\geq 0\\
X_0 &=& \eta \in \cC_r = C([-r,0],\R^d) \nonumber
\end{eqnarray}
or in the integral form
\begin{eqnarray}\label{eq02}
X(t) &=& \eta(0)+ \int_0^tF(X_s)ds + \int_0^t G(X_s)dB^H(s), \  t\geq 0\\
X_0 &=& \eta \in \cC_r  \nonumber
\end{eqnarray}
where $r>0$ is the finite delay and $\cC_r$ is the space of continuous functions $\eta$ from $[-r,0]$ to $\R^d$ endowed with the uniform norm
$\|\eta\|_{\infty} = \max \limits_{s\in [-r,0]} \|\eta(s)\|$, $X_s \in \cC_r$ denotes the function defined by $X_s(\cdot) = X(s+\cdot)$, and $F,G: \cC_r \to \R^d$ are coefficient functions.

We consider the following assumptions on the coefficients of our fSDDE \eqref{eq01}.

\paragraph{($H_F$)} The function $F$ is globally Lipschitz continuous and thus has linear growth, i.e.\ there exist constants $L_1,L_2 >0$ such that for all $\xi, \eta \in \cC_r$
\[
\|F(\xi) - F(\eta)\| \leq L_1 \|\xi-\eta\|_{\infty}\ \quad\text{and}\quad \|F(\xi)\| \leq L_2(1+ \|\xi\|_{\infty}).
\]

\paragraph{($H_G$)} The function $G$ is $C^1$ such that its Frechet derivative w.r.t.\ $\xi$ is bounded and globally Lipschitz continuous, i.e.\ there exist constants $L_3,L_4 >0$ such that for all $\xi, \eta \in \cC_r$
\[
\|D^1 G(\xi)\| \leq L_3 \quad\text{and}\quad \|D^1 G(\xi) - D^1 G(\eta)\| \leq L_4 \|\xi-\eta\|_{\infty}.
\]

%%%%%%%%%%%%%%%%%%%%%%%%%%%%%%%%%%%%%%%%%%%%%%

\begin{theorem}\label{mainthm}
Assume that $F,G$ satisfy the assumptions $H_F$ and $H_G$. Fix $\alpha \in (1-H, \frac{1}{2})$ and $T>0$. Then for each $\eta \in C^{0,1-\alpha}([-r,0],\R^d)$, there exists a unique solution $X(t,\omega,\eta)$ of \eqref{eq02} such that $X(\cdot,\omega,\eta) \in C^{0,1-\alpha}([-r,T],\R^d)$ almost surely. The solution generates a continuous random dynamical system $\varphi: \R_+ \times \Omega \times C^{0,1-\alpha}([-r,0],\R^d) \to C^{0,1-\alpha}([-r,0],\R^d)$ which is given by
\begin{equation}\label{cocycle}
\varphi(t,\omega,\eta) = X_t(\cdot,\omega,\eta), \ \forall t\geq 0.
\end{equation}
\end{theorem}

%%%%%%%%%%%%%%%%%%%%%%%%%%%%%%%%%%%%%%%%%%%%%%

\begin{proof}
The proof will be divided into several steps.

{\bf Step 1}: The existence and uniqueness part  of \cite[Proposition 4.1]{bou1} as stated in the following proposition.

\begin{proposition}\label{prop1}
Fix $\nu \in (\frac{1}{2}, H)$ and $\alpha \in (1-\nu,\frac{1}{2})$. For $\omega \in C^{\nu}([0,T],\R)$ and $x \in C^{0,1-\alpha}([-r,T],\R^d)$, denote
\[
I(x)(t) = \int_0^t F(x_s) ds \quad\text{and}\quad J(x)(t) = \int_0^t G(x_s)d\omega(s).
\]
Then, under assumption $H_F$ and $H_G$, we have $I(x), J(x) \in C^{0,1-\alpha}([0,T],\R^d)$.
\end{proposition}

\begin{proof}[Proof of Proposition \ref{prop1}]
The proof is a direct consequence of inequalities (4), (5) and (6) in the proof of \cite[Proposition 4.1]{bou1}, with $T$ replaced by $(t-s)$. In fact, for the Lebesgue integral we have
\begin{eqnarray*}
\|I(x)\|_{1-\alpha,\delta,0,T}=\sup \limits_{\substack{0\leq s < t\leq T,\\ |t-s|< \delta }}\frac{\|I(x)(t)-I(x)(s)\|}{|t-s|^{1-\alpha}} &\leq&  L_2  (1+ \|x\|_{1-\alpha}) \delta^{\alpha}.
\end{eqnarray*}
Thus $I(x)\in C^{0,1-\alpha}([0,T],\R^d)$.

In addition, for $\alpha \in (1-\nu,1/2)$ by choosing $\beta \in (1-\nu,\alpha)$, we then get $C^{1-\alpha}([0,T],\R^d) \subset W^{\beta,1}_0([0,T],\R^d)$ and $\omega \in W^{1-\beta,\infty}_T([0,T],\R^d)$ due to \eqref{prop2}. For a universal constant $c>0$ which may change from line to line we obtain
\begin{align}\label{eq100}
\begin{split}
\|J(x)(t)&-J(x)(s)\| \\
&\leq  c\Lambda_\beta(\omega)\int_s^t\bigg(\frac{\|G(x_r)\|}{(r-s)^{\beta}}
+\int_s^r\frac{\|G(x_r)-G(x_q)\|}{(r-q)^{1+\beta}}dq\bigg)dr\\
&\le c\Lambda_\beta(\omega)\int_s^t\bigg(\frac{\max(L_3,\|G(0)\|)(1+\|x\|_{1-\alpha})}{(r-s)^{\beta}}+
\int_s^r\frac{L_3\|x\|_{1-\alpha}(r-q)^{1-\alpha}}{(r-q)^{1+\beta}}dq\bigg)dr\\
&\le c\Lambda_\beta(\omega)(1+\|x\|_{1-\alpha})(t-s)^{1-\beta}\\
&\le c\Lambda_\beta(\omega)(1+\|x\|_{1-\alpha})(t-s)^{1-\alpha} (t-s)^{\alpha-\beta}.
\end{split}
\end{align}
Thus
\[
\|J(x)\|_{1-\alpha,\delta,0,T}=\sup \limits_{\substack{0\leq s < t\leq T,\\ |t-s|< \delta }}\frac{\|J(x)(t)-J(x)(s)\|}{|t-s|^{1-\alpha}} \leq c \Lambda_{\beta}(\omega)(1+\|x\|_{1-\alpha}) \delta^{\alpha-\beta}.
\]
Similarly, as we have concluded for $I(x)$, we obtain that $J(x)\in C^{0,1-\alpha}([0,T],\R^d)$.

Since $\eta\in C^{0,1-\alpha}([-r,0],\R^d)$ and the above integrals are in $C^{0,1-\alpha}([0,T],\R^d)$, we can concatenate these mappings
to $X(\cdot,\omega,\eta)\in C^{0,1-\alpha}([-r,T],\R^d)$. Indeed, the concatenation of two elements $\eta \in C^{0,1-\alpha}([-r,0],\R^d)$ and $\mu\in C^{0,1-\alpha}([0,T],\R^d) $ with $\eta(0)=\mu(0)$ is an element $\xi$ of $C^{0,1-\alpha}([-r,T],\R^d)$ due to the estimate
\begin{eqnarray*}
\|\xi\|_{1-\alpha,\delta,-r,T} &=& \max \Big\{\sup \limits_{\substack{-r\leq s< t\leq 0,\\ |t-s|< \delta }}\frac{\|\xi(t)-\xi(s)\|}{|t-s|^{1-\alpha}},\sup \limits_{\substack{0\leq s < t\leq T,\\ |t-s|< \delta }}\frac{\|\xi(t)-\xi(s)\|}{|t-s|^{1-\alpha}} ,\\
 &{}&\sup \limits_{\substack{-r\leq s\leq 0 < t\leq T,\\ |t-s|< \delta }}\frac{\|\xi(t)-\xi(s)\|}{|t-s|^{1-\alpha}}\Big\}\\
&=&\max \Big\{\|\xi\|_{1-\alpha,\delta,-r,0}, \|\xi\|_{1-\alpha,\delta,0,T}, \sup \limits_{\substack{-r\leq s\leq 0 < t\leq T,\\ |t-s|< \delta }}\frac{\|\xi(t)-\xi(s)\|}{|t-s|^{1-\alpha}}\Big\}\\
&\leq& \max \Big\{\|\eta\|_{1-\alpha,\delta,-r,0}, \|\mu\|_{1-\alpha,\delta,0,T}, \\
&{}& \sup \limits_{\substack{-r\leq s\leq 0 < t\leq T,\\ |t-s|< \delta }}\frac{\|\xi(t)-\xi(0)\| + \|\xi(0)-\xi(s)\| }{|t-s|^{1-\alpha}}\Big\}\\
&\leq& \max \Big\{\|\eta\|_{1-\alpha,\delta,-r,0}, \|\mu\|_{1-\alpha,\delta,0,T},\\
&{}& \sup \limits_{\substack{-r\leq s\leq 0 < t\leq T,\\ |t|,|s|,|t-s|< \delta }}\frac{\|\mu\|_{1-\alpha,\delta,0,T} |t|^{1-\alpha} + \|\eta\|_{1-\alpha,\delta,-r,0} |s|^{1-\alpha}}{|t-s|^{1-\alpha}} \Big\}\\
&\leq& \|\mu\|_{1-\alpha,\delta,0,T} + \|\eta\|_{1-\alpha,\delta,-r,0}.
\end{eqnarray*}
\end{proof}

With the help of Proposition \ref{prop1}, the proof is then completely analog to the proof of Theorem 2.1 in Boufoussi and Hajji \cite{bou1}, with spaces $C^{1-\alpha}([-r,T],\R^d)$ being replaced by $C^{0,1-\alpha}([-r,T],\R^d)$.

{\bf Step 2}: Assume that $X(t,\omega,\eta)$ is the unique solution of \eqref{eq02} with fixed initial value $\eta \in C^{0,1-\alpha}([-r,0],\R^d)$. Define $\varphi$ as in \eqref{cocycle}. To prove the cocycle property, we use Lemma 5 in Garrido-Atienza et al.\ \cite{atienza1}, which states that for $a,b,c \in \R$ such that $a,b,a-c,b-c \in [0,T]$ and $f\in W^{\alpha,1}_0([0,T],\R^d), \omega \in W^{1-\alpha,\infty}_T([0,T],\R)$,
\[
\int_a^b f(s)d\omega(s) = \int_{a-c}^{b-c} f(s+c) d \theta_c \omega(s).
\]
Then for fixed $t,\tau \geq 0$, $\omega \in \Omega$, and $s \in [-r,0]$, we consider two cases:
\begin{itemize}
\item Case 1: $t+s \geq 0$, then $t+\tau+s\geq 0$. We get
\begin{eqnarray*}
&& \varphi(t+\tau,\omega,\eta)(s) = X_{t+\tau}(s,\omega,\eta) = X(t+\tau+s,\omega,\eta)
\\
&=& \eta(0) + \int_0^{t+\tau+s} F(X_u(\cdot,\omega,\eta))du + \int_0^{t+\tau+s} G(X_u(\cdot,\omega,\eta))d \omega(u)
\\
&=& \eta(0) + \int_0^{\tau} F(X_u(\cdot,\omega,\eta))du + \int_{\tau}^{t+\tau+s} F(X_u(\cdot,\omega,\eta))du + {}
\\
&& {} + \int_0^{\tau} G(X_u(\cdot,\omega,\eta))d \omega(u) + \int_{\tau}^{t+\tau+s} G_u(X(\cdot,\omega,\eta))d \omega(u)
\\
&= & X(\tau,\omega,\eta) + \int_{0}^{t+s} F(X_{u+\tau}(\cdot,\omega,\eta))du + \int_{0}^{t+s} G(X_{u+\tau}\cdot,\omega,\eta))d \theta_{\tau}\omega(u)
\\
& = & X_{\tau}(0,\omega,\eta) + \int_0^{t+s} F(X_{u}(\tau+\cdot,\omega,\eta))du + \int_{0}^{t+s} G(X_u(\tau+\cdot,\omega,\eta))d \theta_{\tau}\omega(u).
\end{eqnarray*}
If we define $Y(\cdot,\omega,\eta) = X(\cdot+\tau,\omega,\eta)$ then we can read off by the uniqueness of the solution of \eqref{eq01} from the right hand side of the last equation that $Y(\cdot,\omega,\eta)$ is the solution of
\begin{eqnarray*}
dY(t) &=& F(Y_t)dt + G(Y_t)dB^H(t,\theta_{\tau}\omega)\\
Y_0 &=& X_{\tau}(\cdot,\omega,\eta) \in C^{0,1-\alpha}([-r,0],\R^d).
\end{eqnarray*}
Thus it follows from the existence and uniqueness of the solution that
\begin{eqnarray*}
\varphi(t+\tau,\omega,\eta)(s) &=& X(t+s,\theta_{\tau} \omega, X_{\tau}(\cdot,\omega,\eta))
= X_t(s,\theta_{\tau} \omega, X_{\tau}(\cdot,\omega,\eta)) \\
&=& \varphi(t,\theta_{\tau} \omega, \varphi(\tau,\omega,\eta))(s).
\end{eqnarray*}

\item Case 2: $t+s <0$. Then by definition
\begin{eqnarray*}
\varphi(t,\theta_{\tau} \omega, \varphi(\tau,\omega,\eta))(s) &=& X(t+s,\theta_{\tau}\omega,X_{\tau}(\cdot,\omega,\eta)) = X_{\tau}(t+s,\omega,\eta) \\
&=& X(\tau + t+ s, \omega,\eta) = X_{t+\tau}(s,\omega,\eta)\\
&=& \varphi(t+\tau,\omega,\eta)(s).
\end{eqnarray*}
\end{itemize}
Thus the cocycle property of $\varphi$ is proved.

{\bf Step 3}: In order to prove the measurability of $\varphi$, we are going to show that $\varphi$ is continuous w.r.t.\ the argument $(t,\eta)$ and measurable w.r.t.\ $\omega$.  First, the continuity of $\varphi(t,\omega,\eta)$ in $ C^{0,1-\alpha}([-r,0],\R^d)$ for  $\eta \in C^{0,1-\alpha}([-r,0],\R^d)$ is a consequence of Proposition 5.4 in Boufoussi and Hajji \cite{bou1}. In fact, it follows from \cite[Proposition 5.4]{bou1} that for fixed $t,\omega$, $\varphi(t,\omega,\eta)$ is locally Lipschitz continuous w.r.t.\ $\eta$ and the local Lipschitz constant depends on $\alpha, T$ and is independent of $t$.

Second, to prove the continuity of $\varphi$ w.r.t.\ $t$ at $\tau$, notice that
\begin{eqnarray}\label{eq11}
&&\|\varphi(t,\omega,\eta)- \varphi(\tau,\omega,\eta)\|_{1-\alpha} = \|X_t(\cdot,\omega,\eta)- X_{\tau}(\cdot,\omega,\eta)\|_{1-\alpha} \\ \nonumber
&&= \sup \limits_{s\in [-r,0]} \|X(t+s,\omega,\eta) - X(\tau+s,\omega,\eta)\|\\ \nonumber
&&+ \sup \limits_{-r\leq s<u\leq 0} \frac{\|X(t+u,\omega,\eta) - X(\tau+u,\omega,\eta) - X(t+s,\omega,\eta) + X(\tau+s,\omega,\eta)\|}{|u-s|^{1-\alpha}}.
\end{eqnarray}
The first term in \eqref{eq11} can be estimated by
\[
\sup \limits_{s\in [-r,0]} \|X(t+s,\omega,\eta) - X(\tau+s,\omega,\eta)\| \leq |t-\tau|^{1-\alpha} \|X(\cdot,\omega,\eta)\|_{1-\alpha}.
\]
To estimate the second term in \eqref{eq11}, write in short $M(u,s)$ for the numerator. Then
\begin{eqnarray*}
&&\sup \limits_{-r\leq s<u\leq 0} \frac{M(u,s)}{|u-s|^{1-\alpha}} \\
&& \leq  \max \Big\{ \sup \limits_{\substack{-r\leq s<u\leq 0,\\|u-s| \leq |t-\tau|^{\frac{1}{2}} }} \frac{M(u,s)}{|u-s|^{1-\alpha}}, \sup \limits_{\substack{-r\leq s<u\leq 0,\\|u-s| \geq |t-\tau|^{\frac{1}{2}} }} \frac{M(u,s)}{|u-s|^{1-\alpha}} \Big\}
\\
&&\leq \max \Big\{ \sup \limits_{\substack{-r\leq s<u\leq 0\\ |u-s|\leq |t-\tau|^{\frac{1}{2}}}} \frac{M(u,s)}{|u-s|^{1-\alpha}}, \sup \limits_{\substack{-r\leq s<u\leq 0\\ |u-s|\geq |t-\tau|^{\frac{1}{2}}}} \frac{M(u,s)}{|u-s|^{1-\alpha}}\Big\} \\
&&\leq \max \Big \{ \sup \limits_{\substack{-r\leq s<u\leq 0\\ |u-s|\leq |t-\tau|^{\frac{1}{2}}}} \frac{\|X(t+u)- X(t+s)\| + \|X(\tau+u)-X(\tau+s)\|}{|u-s|^{1-\alpha}},\\
&& {} \sup \limits_{\substack{-r\leq s<u\leq 0\\ |u-s|\geq |t-\tau|^{\frac{1}{2}}}} \frac{\|X(t+u)-X(\tau+u)\| + \|X(t+s)-X(\tau+s)\|}{|u-s|^{1-\alpha}}\Big\} \\
&& \leq 2 \|X(\cdot,\omega,\eta)\|_{1-\alpha,|t-\tau|^{\frac{1}{2}},-r,T}+ \sup \limits_{\substack{-r\leq s<u\leq 0\\ |u-s|\geq |t-\tau|^{\frac{1}{2}}}} \frac{2 |t-\tau|^{1-\alpha}\|X(\cdot,\omega,\eta)\|_{1-\alpha} }{|t-\tau|^{\frac{1}{2}(1-\alpha)}} \\
&&\leq 2 \|X(\cdot,\omega,\eta)\|_{1-\alpha,|t-\tau|^{\frac{1}{2}},-r,T}+2|t-\tau|^{\frac{1}{2}(1-\alpha)} \|X(\cdot,\omega,\eta)\|_{1-\alpha}.
\end{eqnarray*}
Finally, by combining the two estimates, we get
\begin{eqnarray}\label{eq09}
\|\varphi(t,\omega,\eta)- \varphi(\tau,\omega,\eta)\|_{1-\alpha} &\leq& 2 \|X(\cdot,\omega,\eta)\|_{1-\alpha,|t-\tau|^{\frac{1}{2}},-r,T} \\ \nonumber
&+&\Big(2|t-\tau|^{\frac{1}{2}(1-\alpha)}+|t-\tau|^{1-\alpha}\Big) \|X(\cdot,\omega,\eta)\|_{1-\alpha}.
\end{eqnarray}
Observe that by the existence and uniqueness in Step 1, $X(\cdot,\omega,\eta) \in C^{0,1-\alpha}([-r,T],\R^d)$. Hence, as $t\to \tau$, the right hand side of \eqref{eq09} tends to $0$, which implies that $\varphi(t, \omega,\eta)$ is continuous w.r.t.\ $t$ at $\tau$ when $\omega, \eta$ fixed.

Furthermore, for fixed $\omega$
\begin{equation}\label{eq08}
\|\varphi(t_1,\omega,\eta_1) - \varphi(t,\omega,\eta)\|_{1-\alpha}  \leq \|\varphi(t_1,\omega,\eta_1) - \varphi(t_1,\omega,\eta)\|_{1-\alpha} + \|\varphi(t_1,\omega,\eta) - \varphi(t,\omega,\eta)\|_{1-\alpha}.
\end{equation}
By introducing the product space $[0,T] \times C^{0,1-\alpha}([-r,0],\R^d)$ with the metric
\[
\|(t_1,\eta_1) - (t,\eta)\| = \sqrt{|t_1-t|^2 + \|\eta_1 - \eta\|_{1-\alpha}^2},
\]
we see that if $(t_1,\eta_1) \to (t,\eta)$ then $t_1 \to t$ and $\eta_1 \to \eta$. Then the first term in the right hand side of inequality \eqref{eq08} tends to 0 because the continuity of $\varphi$ in $\eta$ is uniform w.r.t.\ $t$. Meanwhile the second term in the right hand side of inequality \eqref{eq08} also tends to $0$ as $t_1 \to t$ due to the continuity of $\varphi$ w.r.t.\ $t\in [0,T]$. Therefore, $\varphi(t,\omega,\eta)$ is continuous w.r.t.\ $(t,\eta)$ when $\omega$ is fixed.

Third, to prove that $\varphi(t,\cdot,\eta)$ is measurable with respect to $\omega$, we follow the proof of Theorem 5.1 in Boufoussi and Hajji \cite{bou1} by introducing the closed set $C^{0,1-\alpha}([-r,T],\eta) = \{f \in C^{0,1-\alpha}([-r,T],\R^d) \,|\, f_{|_{[-r,0]}}(\cdot) = \eta(\cdot)\}$ and the map $U:\Omega \times C^{0,1-\alpha}([-r,T],\eta) \to C^{0,1-\alpha}([-r,T],\eta)$ by
\begin{eqnarray*}
U_{\omega}(f)(t) = \left\{ \begin{array}{ll}
         \eta(t) & \mbox{if $t \in [-r,0]$},\\
        \eta(0) + \int_0^t F(f_s)ds + \int_0^t G(f_s)d \omega(s) & \mbox{if $t \geq 0$}.\end{array} \right.
\end{eqnarray*}

Using the equivalent norm in $C^{1-\alpha}([-r,T],\R^d)$ given by
\[
\|f\|_{1-\alpha,\lambda} = \sup \limits_{t\in [-r,T]} e^{-\lambda t} \|f(t)\| + \sup \limits_{-r\leq s<t\leq T} e^{-\lambda t} \frac{\|f(t)-f(s)\|}{|t-s|^{1-\alpha}},
\]
as proved in Theorem 5.1 in Boufoussi and Hajji \cite{bou1}, we can choose $\lambda = \lambda_0$ large enough and a constant $M_0$ independent of $f$ and $\omega$ such that with
\[
\mathcal{B}_{\lambda_0,\eta} := \{f \in C^{0,1-\alpha}([-r,T],\eta) \,:\, \|f\|_{1-\alpha,\lambda_0} \leq M_0\}
 \]
the map $U_{\omega}(f) \in \mathcal{B}_{\lambda_0,\eta}$ and such that $U_{\omega}(\cdot): \mathcal{B}_{\lambda_0,\eta} \to \mathcal{B}_{\lambda_0,\eta}$ is a contraction map, i.e.
\begin{equation}\label{contraction}
\|U_{\omega}(f) - U_{\omega}(g)\|_{1-\alpha,\lambda_0} \leq k_{\lambda_0,\omega} \|f-g\|_{1-\alpha,\lambda_0}
\end{equation}
with the contraction coefficient $k_{\lambda_0,\omega}<1$. For each  $f \in \mathcal{B}_{\lambda_0,\eta}$ and $\omega^1, \omega^2 \in W^{1-\alpha,\infty}_T([0,T],\R)$, since $f \in C^{0,1-\alpha}([0,T],\R^d) \subset W^{\alpha,1}_0([0,T],\R^d)$ and $\omega^1 - \omega^2 \in W^{1-\alpha,\infty}_T([0,T],\R)$, we can apply Proposition 4.1 in Nualart and R{\u{a}}{\c{s}}canu \cite{nualart3} to get
\begin{eqnarray}\label{Lipschitz}
\|U_{\omega^1}(f) - U_{\omega^2}(f)\|_{1-\alpha,\lambda_0} &=& \|U_{\omega^1-\omega^2}(f) \|_{1-\alpha,\lambda_0}\\\nonumber
&\leq& e^{r \lambda_0}\|U_{\omega^1}(f) - U_{\omega^2}(f)\|_{1-\alpha} \\ \nonumber
&\leq& \frac{e^{r \lambda_0}}{\Gamma(1-\alpha) \Gamma(\alpha)} \|\omega^1 - \omega^2\|_{1-\alpha,\infty,T} \|G(f(\cdot))\|_{\alpha,1}
\\ \nonumber
&\leq & \frac{e^{r \lambda_0}\max\{L_3,L_4\} (1+ \|f\|_{\alpha,1})}{\Gamma(1-\alpha) \Gamma(\alpha)}\|\omega^1 - \omega^2\|_{1-\alpha,\infty,T} \\ \nonumber
&\leq & \frac{e^{r \lambda_0} \max\{L_3,L_4\} (1+\sup \limits_{f \in \mathcal{B}_{\lambda_0,\eta}} \|f\|_{\alpha,1} )}{\Gamma(1-\alpha) \Gamma(\alpha)}\|\omega^1 - \omega^2\|_{1-\alpha,\infty,T} \\ \nonumber
&\leq& L_N \|\omega^1 - \omega^2\|_{1-\alpha,\infty,T},
\end{eqnarray}
where $L_N>0$ is independent of $f,\omega$. Hence, $U_{\omega}(f)$ is continuous and hence measurable with respect to $\omega$.

From \eqref{contraction} and \eqref{Lipschitz}, it is easy to see that the sequence
\[
f_n(\omega) := \underbrace{U_{\omega} \circ \dots \circ  U_{\omega}}_{n \ \text{times}}(f),\ n \geq 1,
\]
is measurable w.r.t.\ $\omega$ and converges to $X(\cdot,\omega,\eta)$ which is the fixed point of $U_{\omega}$. Therefore, we conclude that the unique solution $X(\cdot,\omega,\eta)$ is also measurable w.r.t.\ $\omega$. That implies that $\varphi(t,\omega,\eta) = X_t(\cdot,\omega,\eta) = X(t+ \cdot,\omega,\eta)$ is also measurable w.r.t.\ $\omega$.

Finally, since the product space $[0,T] \times C^{0,1-\alpha}([-r,0],\R^d)$ is separable, we apply Lemma III.14 in Castaing and Valadier \cite{castaing} to conclude that $\varphi$ is measurable w.r.t.\ $(t,\omega,\eta)$ which proves $\varphi$ to be a random dynamical system. The continuity w.r.t.\ $(t,\eta)$ then confirms $\varphi$ to be a continuous random dynamical system.
\end{proof}

%%%%%%%%%%%%%%%%%%%%%%%%%%%%%%%%%%%%%%%%%%%%%%

\begin{remark}\label{rem1}
It is important to notice here that even though the solution $X(t,\omega,\eta)$ is continuous w.r.t.\ $t$, the cocycle $\varphi$ might not be continuous w.r.t.\ $t$ if we consider $\eta$ in $C^{1-\alpha}([-r,0],\R^d)$ instead of $C^{0,1-\alpha}([-r,0],\R^d)$. For a counterexample, choose $r = 1$ and consider the equation
\begin{eqnarray}\label{eq05}
dX(t) &=& 0, \ \forall t \geq 0, \\ \nonumber
X(\cdot) &=& \eta(\cdot) \in C^{1-\alpha}([-1,0], \R),
\end{eqnarray}
where $\eta(t) := |t|^{1-\alpha} \forall t\in [-1,0]$. It is easy to check that $\eta \in C^{1-\alpha}([-1,0],\R)$ with $\|\eta\|_{1-\alpha} = 2$. System \eqref{eq05} has a unique solution $X(t) = \eta(0)$ for $t \geq 0$ and $X(t) = \eta(t)$ for $t\in [-1,0]$. For fixed $t,\tau\in [0,1], t\ne \tau$ such that $|t-\tau| \leq 1$, by choosing $u = -\tau, s = -t \in [-1,0]$ in the right hand side of \eqref{eq11}, we get
\begin{eqnarray*}
\|\varphi(t,\omega,\eta) - \varphi(\tau,\omega,\eta)\|_{1-\alpha}&\geq& \frac{\|X(t-\tau,\omega,\eta) - X(0,\omega,\eta) - X(0,\omega,\eta)+ X(\tau-t,\omega,\eta)\|}{|t-\tau|^{1-\alpha}} \\
&\geq& \frac{\|X(|t-\tau|,\omega,\eta) + X(-|t-\tau|,\omega,\eta) - 2 X(0,\omega,\eta)\|}{|t-\tau|^{1-\alpha}}\\
&\geq& \frac{\|\eta(-|t-\tau|) - \eta(0)\|}{|t-\tau|^{1-\alpha}} \\
&\geq&\frac{\||t-\tau|^{1-\alpha} - 0\|}{|t-\tau|^{1-\alpha}} = 1
\end{eqnarray*}
Hence $\varphi$ is not continuous w.r.t. $t$ at any $\tau \in [0,1]$.

%ii, By the same arguments and technique, we can prove all conclusions in Theorem %\ref{mainthm} where the hypotheses of $F,G$ are replaced by weaker conditions using %H\"older $\nu-$ norm, as follows: For fixed $\alpha \in (1-H,\frac{1}{2})$, consider the %space $C^{0,1-\alpha}([-r,0])$. Then the two functions $F,G: C^{0,1-\alpha}([-r,0]) \to %\R^d$ satisfies
%\begin{eqnarray}
%&1,& \|F(\xi)- F(\eta)\| \leq L_1 \|\xi-\eta\|_{1-\alpha}, \\ \nonumber
%&2,& \|F(\xi)\| \leq L_2 (1+ \|\xi\|_{1-\alpha}), \\ \nonumber
%&3,& \|D^1G(\xi)\|_{\cL(C^{0,1-\alpha}([-r,0],\R^d),\R^d)} \leq L_3, \\ \nonumber
%&4,& \|D^1G(\xi)-D^1G(\eta)\|_{\cL(C^{0,1-\alpha}([-r,0],\R^d),\R^d)} \leq L_4 %\|\xi-\eta\|_{1-\alpha}.
%\end{eqnarray}
\end{remark}

%%%%%%%%%%%%%%%%%%%%%%%%%%%%%%%%%%%%%%%%%%%%%%%%%%%%%%%%%%%%%%%%

\end{document}